\documentclass[11pt,a4paper]{article}

\usepackage[sort,numbers]{natbib}
\usepackage{amsfonts} 
\usepackage{mathtools}
\usepackage{hyperref}
\usepackage{amssymb}
\usepackage{amsthm, amsmath}
\usepackage{tikz}
\usepackage{csquotes}
\usepackage{tikz-cd}
\usepackage[affil-it]{authblk}
\usepackage{blkarray}
\usepackage[english]{babel}
\usepackage[ruled,vlined]{algorithm2e}
\usepackage[a4paper, total={6in, 8in}]{geometry}

\usepackage{tikz-cd}
\usepackage{nicematrix}
\usepackage{bm}

\theoremstyle{definition}
\newtheorem{theorem}{Theorem}[section]
\newtheorem{theorem*}{Theorem}
\newtheorem{lemma*}{Lemma}
\newtheorem{problem}{Problem}
\newtheorem{proposition}[theorem]{Proposition}
\newtheorem{definition}[theorem]{Definition}
\newtheorem{lemma}[theorem]{Lemma}
\newtheorem{example}[theorem]{Example}
\newtheorem{corollary}[theorem]{Corollary}
\newtheorem{remark}[theorem]{Remark}
\newtheorem{conjecture}[theorem]{Conjecture}

\AtBeginDocument{%
  \providecommand\BibTeX{{%
    \normalfont B\kern-0.5em{\scshape i\kern-0.25em b}\kern-0.8em\TeX}
   \theoremstyle{definition}
\newtheorem{theorem}{Theorem}[section]
\newtheorem{theorem*}{Theorem}
\newtheorem{lemma*}{Lemma}

\newtheorem{proposition}[theorem]{Proposition}
\newtheorem{definition}[theorem]{Definition}
\newtheorem{lemma}[theorem]{Lemma}
\newtheorem{example}[theorem]{Example}
\newtheorem{corollary}[theorem]{Corollary}
\newtheorem{remark}[theorem]{Remark}

    }}

\def\restrict#1{\raise-.5ex\hbox{\ensuremath|}_{#1}}

\DeclareMathSymbol{\ast}{\mathbin}{symbols}{"03}

\DeclareMathOperator{\myb}{\mathfrak{m}_y^{b}}
\DeclareMathOperator{\mxinf}{\mathfrak{m}_x^{\infty}}
\DeclareMathOperator{\my}{\mathfrak{m}_y}
\DeclareMathOperator{\mx}{\mathfrak{m}_x}

\DeclareMathOperator{\xreg}{xreg}

\DeclareMathOperator{\bigin}{bigin}

\DeclareMathOperator{\HF}{HF}

\DeclareMathOperator{\HP}{HP}

\DeclareMathOperator{\FGLM}{FGLM}

\DeclareMathOperator{\Length}{length}

\newcommand{\x}{\bm{x}}
\newcommand{\y}{\bm{y}}
\newcommand{\RXh}{R_{X,h}}

\renewcommand{\P}{\mathbb{P}}
\newcommand{\proj}{\mathbb{P}}

\newcommand{\NN}{\mathbb{N}}
\newcommand{\ZZ}{\mathbb{Z}}

\usepackage{xspace}
\newcommand{\groebner}{Gr\"obner\xspace}
\newcommand{\gbs}{\groebner bases\xspace}
\newcommand{\gb}{\groebner basis\xspace}

\RequirePackage{fix-cm}

\usepackage{enumitem}

\begin{document}

\title{Solving bihomogeneous polynomial systems with a zero-dimensional projection
}

\author{Mat\'ias Bender\footnote{Inria \& CMAP, CNRS, \'Ecole Polytechnique, IP Paris, Palaisseau, France, {matibender@inria.fr}} \hspace{0.1cm} 
Laurent Bus\'e\footnote{Universit\'e C\^ote d'Azur, Inria, Sophia Antipolis, France. laurent.buse@inria.fr}
\hspace{0.1cm}
Carles Checa\footnote{University of Copenhagen, Copenhagen, Denmark. {ccn@math.ku.dk}}
\hspace{0.1cm} Elias Tsigaridas\footnote{Inria Paris \& IMJ-PRG, Sorbonne Universit\'e, Paris, France. elias.Tsigaridas@inria.fr} 
}


\date{\today}

\maketitle

\begin{abstract}
We study bihomogeneous systems defining, non-zero dimensional, biprojective varieties for which the projection onto the first group of variables results in a finite set of points.
To compute (with) the 0-dimensional projection and the corresponding quotient ring, we introduce
linear maps that greatly extend the classical multiplication maps for zero-dimensional systems, but are not those associated to the elimination ideal; we also call them multiplication maps.
We construct them using linear algebra on the restriction of the ideal to a carefully chosen bidegree or, if available, from an arbitrary \gb. 
The multiplication maps allow us to compute the elimination ideal of the projection, by generalizing FGLM algorithm to bihomogenous, non-zero dimensional, varieties.
We also study their properties, like their minimal polynomials and the multiplicities of their eigenvalues, and show that we can use the eigenvalues to compute numerical approximations of the zero-dimensional projection.
Finally, we establish a single exponential complexity bound for computing multiplication maps and \gbs, that we express in terms of the bidegrees of the generators of the corresponding bihomogeneous ideal.
\end{abstract}

\maketitle

\section{Introduction}

Solving systems of polynomial equations is a common problem in (applied) mathematics and engineering. Most available computational algebra tools focus on systems that have a finite number of solutions, that is, zero-dimensional polynomial systems.
We present new mathematical and algorithmic tools
to study (and to compute the solutions of) non-zero dimensional bihomogeneous systems for which the projection onto the first group of variables is zero dimensional. 
We introduce a novel technique to recover efficiently the points in the zero dimensional projection.

Let us introduce the setting of the problem
and our techniques by considering different algebraic formulations and solving strategies of a well-known problem from linear algebra:
the computation of eigenvalues of a matrix.
Needless to say that we do not promote computing in this way eigenvalues, but the problem serves as an (excellent) working example of our formulation.

\paragraph{Computing eigenvalues} 
Let $k$ be an infinite field and $A \in k^{m \times m}$ a matrix.
We can consider the computation of the eigenvalues of $A$ as an elimination problem. In particular, if we consider the variety 
\begin{equation}
    \label{eq:eigen-variety}
V := \{ (\bm{x}, \bm{y}) \in \P^1 \times \P^{m-1} : (x_0 A - x_1 Id) \cdot \bm{y}^{\top} = 0 \} \subset  \P^1 \times \P^{m-1} ,
\end{equation}
then the eigenvalues correspond to computing the projection $\pi(V) \subset \mathbb P^1$.
An obstacle of this formulation is that even though the projection of $\pi(V)$ consists of a finite number of points, the variety $V$ can be much more complicated, e.g., it might contain irreducible components of different (possibly high) dimensions.
Therefore, it is not easy to transform the problem into the task of solving a zero dimensional polynomial system, for which our algorithmic arsenal is enormous, e.g., \cite{ABBOTT2020137,bender2021toric,eisenbudhunekevasconcelos, gianni1988grobner,vdH:polexp,kricklogar, marinarimoellermora, mourraintelen}; the bibliography on polynomial systems solving is huge and the cited references are just the tip of the iceberg.

Another formulation consists in computing the eigenvalues
using the characteristic polynomial of $A$. For this
we consider the determinantal ideal associated to the equation $\det(x_0 A - x_1 Id) = 0$, which we solve to obtain the eigenvalues.
However, this approach introduces new obstacles. On the one hand, we have to compute symbolically, that is exactly, the determinant of a matrix; this is a hard and unstable computation. On the other hand, the solution set defined by the determinantal ideal might be more complicated than the original one; indeed, this simple determinantal system corresponding to the characteristic polynomial gives us the solutions, in our case the eigenvalues, with a higher multiplicity (algebraic multiplicity) than their multiplicity in  $\pi(V)$ (degree in the minimal polynomial). 

Fortunately, if we choose to compute the eigenvalues using
the minimal polynomial of $A$, many of the previous obstacles disappear. 
Indeed, the minimal polynomial has the eigenvalues as roots with the "correct" multiplicity and it generates the elimination ideal of the eigenvalue problem, which corresponds to $\pi(V)$. 
However, even if we employ the minimal polynomial to compute the eigenvalues, 
if we opt to compute it using elimination theory, 
then we face the difficulty that we are forced to compute with the polynomials of the original ideal at high enough degrees in the $\bm{x}$'s and $\bm{y}$'s; the degree of $\bm{x}$'s being at least the degree of the minimal polynomial of $A$. Hence, computing the elimination ideal, in our case the minimal polynomial, in this way is suboptimal.
To overcome this obstacle, (multidimensional) linear recurrence relations can be used \cite{linear_recurrence}.

Our approach is inspired by the computation of the minimal polynomial. We use linear recurrence relations
to compute polynomials that encode the (zero-dimensional) projection of non-zero-dimensional bihomogeneous polynomial systems;
alas not always the elimination ideal.
Indeed, for the eigenvalue problem \eqref{eq:eigen-variety}, our algorithm computes the elimination ideal of $\pi(V(I))$ by computing the minimal polynomial of $A$ using linear recurrence relations.

\paragraph{Problem statement}
Consider the bigraded ring $R := k[\x, \y]$ where $\deg(x_i) = (1,0)$ and $\deg(y_j) = (0,1)$.
An ideal $I \subset R$ is bihomogeneous if it is generated by homogeneous bigraded elements in $R$.
Bihomogeneous polynomial systems are the algebraic counter-part of subschemes and subvarieties of $\proj^n \times \proj^m$. 
We focus on bihomogeneous ideals that define biprojective varieties whose projection to $\P^n$ is a finite set of points. Our objective is to recover these points. 
As in the case of eigenvalues, where understanding the difference of geometric and algebraic multiplicity involves extra computations, that is, to distinguish eigenvectors from generalized eigenvectors, we only focus on recovering the points of the projection and not their multiplicity.
More formally, for a bihomogeneous ideal $I \subset R$, we denote by $V(I) \subset \P^n \times \P^m$ its corresponding biprojective variety.
Let $\pi : \P^n \times \P^m \rightarrow \P^n$ denote the projection on the first set of variables.
We only consider bihomogeneous ideals where $\pi(V(I))$ is finite. 
Our goal is to recover the points in $\pi(V(I))$, which,
in our case, translates to 
{
\setlist{leftmargin=5.5mm}
\begin{itemize}
    \item Obtain a \gb for an (affine) ideal $J \subset k[x_1,\dots,x_n]$ such that its variety agrees with  $\pi(V(I))$.
    \item Compute (approximate) the coordinates of the points in~{\small$\pi(V(I))$}.
\end{itemize}
}

\paragraph{Context and contributions}
Bihomogeneous polynomial systems and their corresponding varieties in $\mathbb{P}^n \times \mathbb{P}^m$ appear commonly in theory and practice when we work with parametrized systems, e.g.,~\cite{LAZARD2007636}, since elimination theory is simpler in this setting.
From a more theoretical point of view, 
we can consider them as the first generalization of
projective varieties and as such they have received a significant amount of attention in the last years~\cite{d2013heights,bender2024multigradedcastelnuovomumfordregularitygrobner,bruce,busechardinnemati,FAUGERE2011406,HA2004153}. 
The bihomogeneous varieties we consider correspond to parameterized polynomial systems, where there is only a finite number of admissible parameters. 
This formulation contains as special cases most of the generalizations of the eigenvalue problems~\cite{vermeersch2,vermeersch}.
Even more, for a special class of such systems, we know explicit  bounds for the bitsize of the solutions~\cite{brownawell}.

A key tool for solving zero-dimensional polynomial systems
are multiplication maps (and their associated matrices), as they encode all the information from the quotient ring by a given ideal; see \cite[Ch.~2]{coxlitosh}. 
A typical approach for solving consists in computing
a \gb of the corresponding ideal (usually using a graded order), that in turn we use to construct the multiplication maps. Then, the FGLM algorithm \cite{fglm} employs these maps
to produce the \gb of the ideal with respect to a lexicographical order.
Moreover, from the eigenvalues and the eigenvectors of the multiplication maps we can compute numerical approximations of the solutions \cite{lazard}. 

There are many different (numerical) linear algebra-based algorithms to construct the multiplication maps, e.g.,~\cite{mourraintelen,mourrainborderbases}, 
which we could say that they lead to a paradigm known as degroebnerization of system solving \cite{grobnerfree}. However, until now, to our knowledge, their use was restricted to zero dimensional systems.

We introduce a novel approach to compute $\pi(V(I))$ by generalizing the use of multiplication maps on the non-zero dimensional ideal $I$.
In particular, we construct matrices
that we use to recover the points in $\pi(V(I))$ from their eigenvalues
and to compute a \gb of an ideal defining these points; for the latter, we employ (a generalization of) FGLM.
These matrices are bigger than the multiplication maps of the associated elimination ideal, but we can compute them directly from the ideal $I$, without computing elimination ideals.
In this way, we generalize the classical pipeline for solving polynomial systems using \gbs
that forces us to use a particular order to compute an elimination ideal, 
and instead, we can use any order to compute these multiplication matrices.

To construct our matrices, we consider the restriction of our ideal to a specific bidegree,  we call it \emph{admissible}; see Def.~\ref{def:admissibleDeg}.
We can verify when a bidegree is admissible by performing some rank computations on certain (Macaulay-like) matrices.
Admissible bidegrees are closely related to the Castelnuovo-Mumford regularity of certain modules~\cite{chardinholandator} 
and also to the multigraded Castelnuovo-Mumford regularity~\cite{bender2024multigradedcastelnuovomumfordregularitygrobner, MGS-mgcmr-04}.
We emphasize that admissible degrees, in same cases, allow us to compute $\pi(V(I))$ by working with matrices corresponding to bidegrees smaller than the Castelnuovo-Mumford regularity of the elimination ideal. For example, in the case of the eigenvalue problem, even though $(1,1)$ is always an admissible bidegree, the Castelnuovo-Mumford regularity of the elimination ideal is the degree of the minimal polynomial of the matrix.

Geometrically, the ideal at an admissible degree induces a  zero dimensional scheme that contains $\pi(V(I))$. Even though we never miss points of 
$\pi(V(I))$, in our computations, they might appear with higher (local) multiplicity and also additional points might be present.
However, if we consider a higher admissible degree, then we obtain a better "approximation"
of $\pi(V(I))$.
Eventually, by increasing the bidegree, we can precisely encode the points in $\pi(V(I))$
and their appropriate multiplicities. 
Nevertheless, computing additional points
is not a problem for our approach because we also introduce a tool to verify if they belong to $\pi(V(I))$.

We compute upper bounds for the admissible bidegrees.
When the variety $V(I)$ defines a finite number of points, we show that the multihomogeneous Macaulay bound \cite{mixedgrobnerbasistowards} is an admissible degree.
We were unable to prove (or disprove) that such bound holds in the case of positive dimensional fibers. 
However, we deduce a (loose) bound using generalized Koszul complexes (Theorem ~\ref{thm:generalBoundAdmDeg}) for the general case with no further assumptions on the geometry of $V(I)$.
Using these bounds, we deduce that the arithmetic complexity of our algorithm for 
recovering the points in  $\pi(V(I))$
is single exponential~(Thm~\ref{thm:singleExponentialBound}).

Finally, we relate the admissible bidegrees with the minimal elements appearing in the \gbs of the bihomogeneous ideals under study, building on regulatity
and bigeneric initial ideals~\cite{bender2024multigradedcastelnuovomumfordregularitygrobner}.

\paragraph{Organization of the paper}
In Section \ref{sec:admissible}, we introduce the admissible bidegrees and study their properties.
In Section \ref{sec::multiplicaiton_maps}, we present the construction of the multiplication maps and their main properties; notably, we show how to use them to recover a \gb of the elimination ideal (see Theorem  \ref{fglm}).
In Section~\ref{sec:eigenstructure}, we study the eigenstructure of these maps, generalizing known results for multiplication maps of zero-dimensional systems.
In Section \ref{sec::bounds}, we present upper bounds for the admissible bidegrees and derive single exponential bounds to solve our problem.
Moreover, we discuss how to decide when a candidate point belongs to the projection.
Finally, in Section \ref{sec:grobner}, we compare the region of admissible bidegrees with the the minimal generators of certain bigeneric initial ideals.

\paragraph{Acknowledgments} 
MB and ET were partially funded by the project SOAP, a public grant from the Fondation Math\'ematique Jacques Hadamard.
CC has been supported by the European Union under the Grant Agreement number 101044561, POSALG. Views and opinions expressed are those of the authors only and do not necessarily reflect those of the European Union or European Research Council (ERC). We thank Marc Chardin for discussions on the bounds appearing in Section~\ref{sec::bounds}.

\paragraph{Notation}
Let $k$ be a field and $\bar{k}$ be its algebraic closure;
and $\P^n$ the projective space of dimension $n \in \NN$ over $\bar{k}$.
Let $R = k[x_0,\dots,x_n,y_0,\dots,y_m]$ be the standard $\mathbb{Z}^2$-graded  ring, with $\deg(x_i) = (1,0)$, for $1 \leq i \leq n$, and $\deg(y_j) = (0,1)$, for $0 \leq j \leq m$.
Also $R_X := k[x_0,\dots,x_n]$.
For a homogeneous $h \in R_X$,  the  ring $\RXh$ is the localization of $R_X$ by  he multiplicative set generated by $h$.
Let $\mx := \langle x_0,\dots,x_n \rangle$, $\my := \langle y_0,\dots,y_m \rangle$ be the irrelevant ideals of $\proj^n$ and $\proj^m$, respectively.

For a bihomogeneous polynomial $f \in R$, its bidegree is $$\deg(f) = (\deg_x(f), \deg_y(f)).$$ 
For a bihomogeheous ideal $I \subset R$, $I_{a,b}$ denotes the vector space of homogeneous polynomials of degree $(a,b)$.
Moreover, the Hilbert function of $R/I$ at degree $(a,b) \in \ZZ^2$ is $\HF_{R/I}(a,b) := \dim_k (R/I)_{a,b}$.
We denote by $\HP_{R/I}$ the Hilbert polynomial of $R/I$, i.e., a polynomial such that $\HF_{R/I}(a,b) = \HP_{R/I}(a,b)$ for $a,b \gg 0$.
We define the $R_X$-module $(R/I)_{\ast,b} := \bigoplus_{a} (R/I)_{a,b}$. The local cohomology modules of $I$ (or $R/I$) with respect to $\mx$ will be denoted as  $H^i_{\mx}(I)$.

Finally, $\pi(V(I))$ is the set of points in the projection onto the first factor of $V(I) \subset \proj^n \times \proj^m$.
For each $\xi \in \proj^n$, $\mathfrak{p}_{\xi} \subset R$ denotes the homogeneous ideal of polynomials that vanish at this point. 

\section{Admissible bidegrees}
\label{sec:admissible}

To generalize multiplication matrices for zero-dimensional projections, we first present conditions for their existence in terms of the Hilbert function of $I$.

\begin{definition}
\label{def:admissibleDeg}
Let $I$ be a bihomogeneous ideal. A bidegree $(a,b) \in \mathbb{Z}^2$ is an \textit{admissible bidegree} for $I$, if $I$ is generated in degrees lower or equal (component-wise) to $(a,b)$ and if there exists a polynomial $h$ of bidegree $(1,0)$ such that
    $$\HF_{R/I}(a,b) = \HF_{R/I}(a+1,b) \quad \text{and} \quad \HF_{R/(I,h)}(a,b) = 0.$$
In this case, we call $h$ an \emph{admissible linear form} for $(a,b)$.
\end{definition}

\begin{remark}
\label{rk:AdmDeg} 
    Admissible bidegrees exist, if and only if,
    $\pi(V(I))$ is a zero-dimensional (or empty) variety. This follows from the fact that the intersection with $V(h)$ is empty, if and only if, the projection of this variety to $\proj^n$ is a finite set of points.
\end{remark}

The following theorem shows that, given an admissible bidegree, the bidegrees with bigger first coordinate
and the same second coordinate are also admissible.

\begin{theorem} \label{thm:stabilizationOfAdmDeg}
Let $I$ be a bihomogeneous ideal. If $(a,b)$ is an admissible bidegree with an admissible linear form $h$, then $(a',b)$ is also an admissible bidegree, with the same admissible linear form  $h$, for every $a' \geq a$.
\end{theorem}

Before proving this theorem, we need a technical lemma that will be also used in the sequel.

\begin{lemma} \label{lem:equivDefReg}
    Consider a bihomogeneous ideal $I \subset R$ and a homogeneous polynomial $g \in R$ of degree $(k,0)$. Then, for any $(a,b) \in \mathbb{Z}_{\geq 0}^2$, any two of the following three conditions implies the third:
    \begin{itemize}
        \item[i)] $I_{(a,b)} = (I:g)_{(a,b)}$.
        \item[ii)] $\HF_{R/(I,g)}(a+k,b) = 0$.
        \item[iii)] $\HF_{R/I}(a,b) = \HF_{R/I}(a+k,b)$.
    \end{itemize}
\end{lemma}

\begin{proof}
Consider the following short exact sequence of vector spaces induced by multiplication by $g$
\begin{equation*}\label{short_exact_sequence} 0  \rightarrow (R/(I:g))_{(a,b)} \xrightarrow[]{ \times g} (R/I)_{(a+k,b)} \rightarrow (R/(I,g))_{(a+k,b)} \rightarrow  0\end{equation*}
As the alternate sum of the dimensions of the vector spaces appearing in this sequence must be equal to zero, we deduce that if $(I:g)_{(a,b)} = I_{(a,b)}$ then conditions $ii)$ and $iii)$ are equivalent. To prove that $ii) + iii)$ imply $i)$, the previous exact sequence implies 
$$
 \HF_{R/(I:g)}(a,b) - \HF_{R/I}(a,b) = \\ \HF_{R/I}(a+k,b) - \HF_{R/I}(a,b) - \HF_{R/(I,g)}(a+k,b) = 0.
$$
Hence, $(I:g)_{(a,b)} = I_{(a,b)}$ since $(I:g) \supseteq I$.
\end{proof}

\begin{proof}[Proof of Theorem ~\ref{thm:stabilizationOfAdmDeg}]
    Let $(a,b)$ be an admissible bidegree and $h$ a corresponding admissible linear form. 

    First, we analyze the syzygies of the ideal $(I,h)$. Since by assumption $\HF_{R/(I,h)}(a,b) = 0$, we deduce that the graded $R_X$-module $\overline{M}_b:=(R/(I,h))_{\ast,b}$ (i.e. the component of degree $b$ with respect to the $\y$'s) satisfies $(\overline{M}_b)_a = 0$, and so $(\overline{M}_b)_{a'}=0$ for every $a'\geq a$. 
    As a consequence, the Castelnuovo-Mumford regularity of $\overline{M}_b$ is $<a$. Let us call $b_2(\overline{M}_b)$ the maximum degree shift that appears in a minimal finite free $R_X$ resolution of $\overline{M}_b$ (i.e.~$b_2(\overline{M}_b)$ is the maximum second Betti number of $\overline{M}_b$).

    By property of the Castelnuovo-Mumford regularity (see for instance \cite[\S 1]{Chardin07}), $b_2(\overline{M}_b)-2<a$, hence $b_2(\overline{M}_b)\leq a+1$. It turns out that $b_2(\overline{M}_b)$ is also an upper bound for the degree at which syzygies of $(I,h)$ in degree $b$ with respect to
    the $\y$'s are generated (this follows from the definition of $\overline{M}_b$). That is, any syzygy of $(I,h)$ of degree $(a',b)$ with $a'\geq a+1$ can be $R_X$-generated from syzygies of $(I,h)$ of degree $(a+1,b)$. 

    Now, as explained in the proof \cite[Lemma~1.9]{bayer_criterion_1987}, which requires that $I$ is generated in degree $\leq (a,b)$, there is a correspondence between syzygies of $(I,h)$ in degree $(a'+1,b)$ and elements in $(I:h)$ of degree $(a',b)$. Thus, we deduce that for every $a'\geq a$, $(I:h)_{(a',b)}$ is generated by $(I:h)_{(a,b)}$. But by our assumption, $(I:h)_{(a,b)}=I_{(a,b)}$ (apply Lemma \ref{lem:equivDefReg} with $k=1$), so it follows that  $(I:h)_{(a',b)}=I_{(a',b)}$, which concludes the proof (apply again Lemma \ref{lem:equivDefReg} with $k=a'-a$).
\end{proof}

\begin{remark}
    The previous proof relies on a deeper result: assuming the  Hilbert polynomial of the $R_X$-module $(R/I)_{*,b}$ is a constant, 
    its Castelnuovo-Mumford regularity is $a$, if and only if, $(a{+1},b)$ is an admissible bidegree of $I$ and $({a},b)$ is not. This follows mutatis mutandis from \cite[Lemma~1.8]{bayer_criterion_1987}.
\end{remark}

We conclude that $I$ is saturated with respect to $\mx$ at the admissible bidegrees.

\begin{corollary}
\label{cor:saturatedAtAdmDeg}
If $(a,b)$ is an admissible bidegree and $h$ an admissible linear form, then
$$
I_{a,b} = (I : \mathfrak{m}_x^\infty)_{a,b} = (I : h^\infty)_{a,b} \enspace .
$$
\end{corollary}

\begin{proof}
For every $k > 0$ and $a' \geq a$, we have the inclusion $I_{a',b} \subseteq (I : \mathfrak{m}_x^k)_{a',b} \subseteq (I : h^k)_{a',b}$. As $(a',b)$ is admissible and $h$ is an admissible form, by Lemma~\ref{lem:equivDefReg}, $(I:h)_{a',b} = I_{a',b}$. The proof follows from the identity $(I:h^2) = ((I:h):h)$.
\end{proof}

To prove that generic linear forms lead to admissible forms, we need to characterize the latter. For this, we introduce the colon ideal $J_b$ given by the annihilator of the $R_X$-module $(R/I)_{*,b}$. 
We first show how we can characterize admissible linear forms using $J_b$ and later we explain the relation between the variety defined by $J_b$ and $\pi(V(I))$.

\begin{definition} \label{def:Jb}
    Given $b \geq 0$, we define the ideal $J_b \subset R_X$ as
  $$   J_b := (I : \mathfrak{m}_y^b) \cap R_X = \mathrm{Ann}_{R_X}((R/I)_{*,b}). $$
 Moreover, let $\pi_b(V(I)) :=V(J_b) \subset \mathbb{P}^n$ be the vanishing set of $J_b$.
\end{definition}

\begin{lemma}\label{lem:characterisationOfGeneric}
    Consider a linear form $g \in R_X$ and an admissible bidegree $(a,b)$. We have that $(I:g)_{a,b} = I_{a,b}$,  if and only if,  $g$ is a non-zero divisor of $R_X / (J_b : \mathfrak{m}_x^\infty)$.
\end{lemma}
\begin{proof}
Consider any irredundant primary decomposition $\cap \mathcal{P}$ of $(I : \mathfrak{m}_x^\infty)$.
Note that the primary components of  $I : \mathfrak{m}_x^\infty$ do not contain powers of $\mx$. Moreover, if $\mathcal{P}$ is a primary ideal such that $\mathcal{P} \supset \my^b$, then $R_{a,b} = \mathcal{P}_{a,b}$ and so, $R_{a,b} = \mathcal{P}_{a,b} \subseteq (\mathcal{P}:g)_{a,b} \subseteq R_{a,b}$.

Every primary component $\mathcal{P}$ of $(I : \mathfrak{m}_x^\infty)$ such that $\mathcal{P} \not\supset \myb$ satisfies $R \neq \sqrt{\mathcal{P}:\myb} \supset \sqrt{\mathcal{P}}$. As colons and intersections commute, 
$
(J_b : \mathfrak{m}_x^\infty) = \bigcap_{\mathcal{P}} (\mathcal{P}:\my^b) \cap R_X 
= \bigcap_{\mathcal{P}\not\supset  \my^{b}} (\mathcal{P}:\my^b) \cap R_X 
$,
so we obtain a (possibly redundant) primary decomposition of $(J_b:\mathfrak{m}_x^{\infty})$.

If $g$ is not a zero divisor in $R_X/(J_b:\mx^{\infty})$, then for $\mathcal{P} \not \supset \my^b$, we have $g \notin \sqrt{\mathcal{P}:\myb} \cap R_X$. Hence, $g \notin \sqrt{P} \cap R_X$ which, as $\mathcal{P}$ is primary, implies that $(\mathcal{P}:g) = \mathcal{P}$. Therefore, as $(a,b)$ is admissible,
\begin{multline*}
(I:g)_{a,b} = ((I:\mathfrak{m}_x^\infty) : g)_{a,b}  = \left(\cap_{\mathcal{P \not\supset  \myb}} \mathcal{P}:g \right)_{a,b} = \\ \cap_{\mathcal{P} \not\supset  \myb} (\mathcal{P}:g)_{a,b}  = \cap_{\mathcal{P} \not\supset \myb} \mathcal{P}_{a,b} = (I : \mathfrak{m}_x^\infty)_{a,b} = I_{a,b}.
\end{multline*}
Conversely, if $(I:g)_{a,b} = I_{a,b}$,
    by Lemma~\ref{lem:equivDefReg}, $g$ is an admissible form.   
    As colon ideals commute,
    $ (J_b : g)_{a',0} = (J_b)_{a',0}$, for every $a' \geq a$, by \cite[Lemma 1.4]{bayer_criterion_1987}, $g$ is not a zero-divisor in $R_X/(J_b : \mx^\infty)$.
\end{proof}

We obtain an effective criterion for admissibility of bidegrees.

\begin{corollary} \label{lem:anyGenericIsOKforReg}
    Assume that $V(I)$ is not empty.
    Let $(a,b)$ be a bidegree such that 
    $\HF_{R/I}(a,b) = \HF_{R/I}(a+1,b)$. Then, $(a,b)$ is admissible , if and only if, $\HF_{R/(I,h')}(a,b) = 0$ for every $h'$ in a open Zariski subset of $R_{1,0}$, that is, for a generic linear form $h'\in R_{1,0}$.
\end{corollary}

\begin{proof}
    As $V(I)$ is not empty, $(J_b : \mxinf) \neq R_X$.
    The proof follows from the fact that, as $(J_b : \mxinf) \neq R_X$ and $k$ is infinite, any generic form $g \in R_X$ is not a zero-divisor in $R_X/(J_b : \mxinf)$; see~\cite[pg.~4]{bayer_criterion_1987}.
\end{proof}

\begin{example}
    \label{ex:running1}
   Consider the ideal $I \subset k[x_0,x_1,x_2,y_0,y_1,y_2]$, where:
   \begin{multline*}I = \langle 
2x_0 - x_1 - x_2, \,
y_0 y_2 - y_1 y_2 - y_2^2, \,
x_1 y_2 - x_2 y_2, \,
y_0^2 - y_1^2 - 2y_1 y_2 - y_2^2, \\
x_1 y_0 - x_1 y_1 - x_2 y_2, \,
x_1^2 - x_1 x_2, \,
x_1 y_1^2 - x_2 y_1^2
\rangle.
   \end{multline*}
Note that $(2,2)$ is already an admissible bidegree as $$
\HF_{R/I}(2,2) = \HF_{R/I}(3,2) \: \HF_{R/(I,h)}(2,2) = 0 \enspace, 
$$  for general $h$. The only point in $\pi(V(I))$ is $\xi = [1:1:1]$.
\end{example}
\begin{remark} \label{rmk:admDegButJbNotRegular}
    The Castelnuovo-Mumford regularity of $J_b$ and the   
    admissible bidegrees of $I$ are not obviously related.
    Consider the eigenvalue problem from the introduction; the degree $(1,1)$ is admissible, but the elimination ideal $J_1$ is generated in the degree of the minimal polynomial of the matrix, generally bigger than $1$.
\end{remark}

\subsection{Geometry and admissible bidegrees}

The classical way of constructing a geometric object from a module is to consider its annihilator. For more details on this construction, see \cite[Ch.~3]{eisenbud1995}.
In our setting, we construct $\pi_b(V(I))$ using an admissible bidegree which corresponds to a variety that contains $\pi(V(I))$ and, possibly, a finite set of extra points. 

Notice that for $b$ sufficiently big, $J_b$ is exactly the elimination ideal of $I$. However, $J_b$ might not coincide with the elimination ideal $(I : \mathfrak{m}_y^\infty) \cap R_X$ for some specific $b$. 
As we mentioned, there may be points in $\pi_b(V(I))$ that are not in the projection. Algebraically, this happens if there are associated primes of $I$ that correspond to points in $\mathbb{P}^n$ but contain the ideal $\my$.
Alternatively, the projective scheme defined by $(I : \mathfrak{m}_y^\infty) \cap R_X$ and $J_b$ might have the same points, but the multiplicities in the latter can be higher.

\begin{lemma}
\label{lem:zeroesOfJb}
Let $I \subset R$ be a bihomogeneous ideal and let $(a,b) \in \mathbb{Z}^2$ an admissible bidegree.
    \begin{itemize}
          \item[i)] For every $b' \geq b$, we have: 
        $$\pi(V(I)) \subseteq \pi_{b'}(V(I)) \subseteq \pi_b(V(I)).$$
        \item[ii)] The variety $\pi_b(V(I))$ is zero-dimensional.
    \end{itemize}
\end{lemma}

\begin{proof}
The inclusions in $i)$ follow from the fact that $$J_b \subseteq J_{b'} \subseteq (I:\mathfrak{m}_y^{\infty}) \cap R_X.$$ 
Similarly to what we used in the proof of Theorem ~\ref{thm:stabilizationOfAdmDeg}, the Hilbert polynomial of $(R/I)_{*,b}$ is constant. Thus, its support consists of a finite number of points \cite[Theorem  11.1]{AtMa69}.  As $J_b$ is the annihilator of $(R/I)_{\ast,b}$, using Hilbert polynomial, it is supported in a finite number of points. By \cite[Lemma 10.40.5]{stacks-project}, the vanishing locus of the annihilator of a finite module coincides with its support. Thus, $\pi_b(V(I))$ is finite.
\end{proof}

\begin{example}[Cont.~Example~\ref{ex:running1}]
\label{ex:running2}
Note that both $(2,2)$ and $(2,3)$ are admissible bidegrees. However, while we have $\pi_3(V(I)) = \pi(V(I))$, $\pi_2(V(I)) = \{[1:1:1], [1:0:2]\}$. 
To understand why this happens, consider the following primary decomposition of $I$, which is also minimal and unique,
$$
I = \langle y_0 - y_1 - y_2, x_1 - x_2, x_0 - x_2 \rangle \cap 
\langle y_0^2, y_1^2, y_2, x_1, 2x_0 - x_2 \rangle \\ 
 \cap  \langle y_2, y_0 + y_1, x_2, x_1, x_0 \rangle
$$
The associated prime of the second primary component contains $\my$ and it cannot be an associated prime of $(I:\mathfrak{m}_y^{\infty})$. However, this primary component does not contain $\mathfrak{m}_y^2$ but $\mathfrak{m}_y^3$, so it contributes with the point $[1:0:2]$ to $\pi_2(V(I))$, but not to $\pi_3(V(I))$. Thus, it will not appear in $\pi_b(V(I))$ for $b \gg 0$ (in this case, $b = 3$).  
\end{example}

\section{Multiplication maps}
\label{sec::multiplicaiton_maps}

In this section, we describe the generalization of multiplication maps to the setting of zero-dimensional projections. Throughout this section, we fix an admissible bidegree $(a,b)$ and a set $B$ of elements of degree $(a,b)$ which form a basis for the vector space $(R/I)_{a,b}$.
For each $g \in R_{(k,0)}$, we define the map
$$\bar{m}_g : (R/I)_{a,b} \rightarrow (R/I)_{a+k,b} \quad [f] \mapsto \bar{m}_g([f]) := [g f].$$ 

We can characterize when a multiplication map is invertible.

\begin{lemma}
\label{lem:invertibleIfRank}
Let $g$ be a polynomial of degree $(k,0)$. 
The map $\bar{m}_g$ is invertible, if and only if, $(I,g)_{a+k,b} = R_{a+k,b}$.
In particular, if $h$ is an admissible form,
the map $\bar{m}_{h^{k}}$ is invertible for any $k \geq 0$.
\end{lemma}
\begin{proof}
    The first statement is trivial. For the second one, observe that as $(a,b)$ is admissible, $\HF_{R/I}(a+k,b)=\HF_{R/I}(a,b)$ and, by Corollary ~\ref{cor:saturatedAtAdmDeg}, $I_{a,b} = (I:h^k)_{a,b}$. The proof follows from Lemma~\ref{lem:equivDefReg}.
\end{proof}

Multiplication maps commute in the following way.

\begin{lemma} \label{lem:commutativityOfMaps}
    Let $h$ be an admissible linear form for $I$ and consider homogeneous $f,g \in R_X$ such that $\deg(f) = s$ and $\deg(g) = t$. Then, we have the following identities,
$$
    \bar{m}_{h^{s+t}}^{-1} \circ \bar{m}_{fg}
    = 
    \bar{m}_{h^s}^{-1} \circ \bar{m}_f \circ \bar{m}_{h^t}^{-1} \circ \bar{m}_g  =
    \bar{m}_{h^t}^{-1} \circ \bar{m}_g \circ \bar{m}_{h^s}^{-1} \circ \bar{m}_f
$$
\end{lemma}

\begin{proof}
We will prove the first equality of the lemma, the other one follows by swapping $f$ and $g$.
First, by Lemma~\ref{lem:invertibleIfRank}, the map $\bar{m}_{h^{k}}$ is invertible for any $k \geq 0$.
  For every $q \in R_{a,b}$ and homogeneous $f \in R_X$, we have that
  $f \, q \equiv h^{k} (\bar{m}_{h^{k}}^{-1} \circ \bar{m}_f)(q) \mod I_{a+k,b}$, where $k = \deg{f}$.
  Hence, for every $q \in I_{a,b}$, 
$$
  h^{s+t} \, (\bar{m}_{h^{s+t}}^{-1} \circ \bar{m}_{g \, f})(q) \equiv 
   g \, f \, q \equiv
   g \, h^{s} (\bar{m}_{h^s}^{-1} \circ \bar{m}_f(q)) \\ \equiv
   h^{t}  \, h^{s} (\bar{m}_{h^t}^{-1} \circ \bar{m}_g \circ \bar{m}_{h^s}^{-1} \circ \bar{m}_f)(q) 
   \mod I.
$$
   By Lemma~\ref{lem:equivDefReg},$(I:h^{s+t})_{a,b} = I_{a,b}$. Hence,
   we have that, for any $q \in R_{a,b}$, $(\bar{m}_{h^{s+t}}^{-1} \circ \bar{m}_{g \, f})(q) \equiv (\bar{m}_{h^t}^{-1} \circ \bar{m}_g \circ \bar{m}_{h^s}^{-1} \circ \bar{m}_f)(q) \mod I_{a,b}$, and so the first equality follows.
\end{proof}

In what follows, fix an admissible linear form $h$ and let $(\RXh)_0$ be the zero-degree part of the localization of $R_X$ by the multiplicative set generated by $h$.

\begin{definition}
    Consider $\frac{g}{h^{{k}}} \in (\RXh)_0$, where $k = \deg{g}$. We define the multiplication map by $\frac{g}{h^{k}}$ 
    as $$m_\frac{g}{h^{{k}}} := \bar{m}_{h^k}^{-1} \circ \bar{m}_g.$$ 
    Let $\mathcal{M}_{\frac{g}{h^{{k}}}}$ be the matrix representing this map in the basis $B$.
\end{definition}

\begin{remark}
        By Lemma~\ref{lem:commutativityOfMaps}, the map $m_{g'}$ does not depend on the rational function that we choose as the representative of $g'$. Note that the matrix of the multiplication map can be constructed either using the Schur complement of the Macaulay matrix at degree $(a+k,b)$ \cite{yet} or using a precomputed Gr\"obner basis with respect to an arbitrary order. For the eigenvalue problem that we considered in the introduction, the computation of the multiplication map by $\frac{x_1}{x_0}$ in the appropriate monomial basis recovers the matrix $A$.
\end{remark}

\begin{lemma}
\label{lem:ringHomomorphism}
    The map \eqref{eq:fromElemToMult} is a commutative ring homomorphism,
    \begin{align} \label{eq:fromElemToMult}
        (\RXh)_0 \xrightarrow[]{} k[m_{{g'}} : g' \in (\RXh)_0 ] \quad g' \mapsto m_{g'} 
    \end{align}
\end{lemma}

\begin{proof}
The map is a ring homomorphism as, for any $f',g' \in  (\RXh)_0$, we have that
    $m_1 = Id$ and $m_{g' f'} = m_{g'} \circ m_{f'} = m_{f'} \circ m_{g'}$ (Lemma~\ref{lem:commutativityOfMaps}) and 
    $m_{g' + f'} = m_{g'} + m_{f'}$.
    To prove this last point, let $f' = \frac{f}{h^s}$ and $g' = \frac{g}{h^t}$, for homogenoeus $f,g \in R_X$. Assume with no loss of generality that $s \geq t$.
    Hence,
    \begin{align*}
    m_{f'} + m_{g'} & = 
    m_{f'} + m_{\frac{h^{t-s}}{h^{t-s}}} \circ m_{g'} =
     m_{f'} + m_{\frac{h^{t-s} + g'}{h^{t}}}
    \\ & =
    \bar{m}_{h^s}^{-1} \circ \bar{m}_f +
    \bar{m}_{h^{s}}^{-1} \circ \bar{m}_{h^{s-t} g} \\ &
    =
    \bar{m}_{h^s}^{-1} \circ (\bar{m}_f + \bar{m}_{h^{s-t} g})
    = m_{f'+g'}
    \end{align*}
    where the second to last equality holds by linearity of $\bar{m}_{h^{s}}^{-1}$.
\end{proof}

Our next theorem relates the commutative ring formed by the multiplication maps and a localization of the ideal $J_b$. 

\begin{theorem}
\label{fglm}
 The kernel of the ring homomorphism defined in Equation~\eqref{eq:fromElemToMult}
 corresponds to the zero-degree part of the localization of the ideal $J_b$ (Theorem ~\ref{def:Jb}) in $\RXh$ , i.e.,
 the ideal $(J_b \otimes_{R_X} \RXh)_0 \subseteq (\RXh)_0.$
\end{theorem}
\begin{proof}
    Let $f' \in (\RXh)_0$.
    We have that ${m_{f'}} = 0$ if and only if 
  $\bar{m}_f = 0$, where $f \in R_X$ is any polynomial such that $f' = \frac{f}{h^k}$.
    Equivalently, if and only if
    $f \, q \in I_{a+k,b}$, for any $q \in R_{a,b}$.

    If $\bar{m}_f = 0$, we have that $f h^a y^\beta \in I_{a+k,b}$, for any $y^\beta \in R_{0,b}$, so $f h^a \in (I : \mathfrak{m}^b_y)_{a+k,0} 
    $. Hence, $\frac{f}{h^k} \in 
    (J_b \otimes_{R_X} \RXh)_0$. 

    If $f \in (J_b \otimes_{R_X} \RXh)_0$, 
    $f h^\omega y^\beta \in I_{\omega+k,b}$ for big enough $\omega \geq 0$ and arbitrary $y^\beta \in R_{0,b}$. 
    So, for any $x^\alpha \in R_{a,0}$, we have that 
    $f h^\omega x^\alpha y^\beta \in I_{a+k+\omega,b}$. 
    As $(a,b)$ is admissible, by Corollary ~\ref{cor:saturatedAtAdmDeg}, $(I_{a+k,b} : h^\omega) = I_{a+k,b}$. 
    Hence, $f x^\alpha y^\beta \in I_{a+k,b}$ for any $x^\alpha y^\beta \in R_{a,b}$, and so $\bar{m}_f = 0$.
\end{proof}

If $h$ is an admissible form, it cannot vanish at any point of $\pi_b(V(I))$, as $R_X / J_b$ is zero-dimensional and $h$ is not a zero-divisor in the saturation; see Lemma~\ref{lem:characterisationOfGeneric}.
Hence, the projective scheme of $J_b \subseteq R_X$ is isomorphic to the affine one of 
     { $(J_b \! \otimes_{R_X} \! \RXh)_0 \! \subseteq \! (\RXh)_0$. }
Therefore, to compute $\pi_b(V(I))$, we will use the localization of $J_b$ at $h$.
As a consequence of Theorem ~\ref{fglm}, we can recover a Gr\"obner bases for the ideal $(J_b \otimes_{R_X} \RXh)_0$ by finding polynomial relations between the multiplication matrices as in FGLM \cite{fglm}.
For this, we assume with no loss of generality that $x_0$ is an admissible linear form in $I$. 
When the characteristic of the field is big enough, this is guaranteed after a generic change of coordinates on the $\bm{x}$ variables.

\begin{corollary}
\label{cor:FGLM}
Assume $h = x_0$ is an admissible linear form.
From the multiplication matrices $\mathcal{M}_{\frac{x_1}{x_0}}, \dots, \mathcal{M}_{\frac{x_n}{x_0}}$, the algorithm $\FGLM$ \cite{fglm} recovers any Gr\"obner basis of $(J_b \otimes_{R_X} R_{X,x_0})_0 \subset k[\frac{x_1}{x_0},\dots,\frac{x_n}{x_0}]$ by computing linear relations between the multiplication maps.
\end{corollary}

\begin{example}[Cont.~Example~\ref{ex:running2}] \label{ex:running3}
At the admissible bidegree $(2,2)$, the monomials
$\{x_0^2 y_0^2, x_0^2 y_0 y_1, x_0^2 y_0 y_2, x_0^2 y_1^2\}$ form a basis for $(R/I)_{2,2}$ the multiplication matrices by $z_1 := \frac{x_1}{x_0}$ and $z_2 := \frac{x_2}{x_0}$ are: 
$$
\mathcal{M}_{z_1} = \begin{pmatrix}
1 & 0 & 0 & 0 \\
0 & 1 & 0 & 0 \\
0 & 0 & 1 & 0 \\
0 & 1 & 1 & 0
\end{pmatrix} \quad \mathcal{M}_{z_2} = \begin{pmatrix}
1 & 0 & 0 & 0 \\
0 & 1 & 0 & 0 \\
0 & 0 & 1 & 0 \\
0 & -1 & -1 & 2
\end{pmatrix}
$$
A lexicographical \gb of the ideal $(J_2 \otimes_{R_X} R_{X,x_0})_0$ is $$\langle z_1 + z_2 - 2, z_2^2 - 3z_2 + 2 \rangle. $$
At degree $(2,4)$, the corresponding multiplication matrices for $z_1$ and $z_2$ are of size $5 \times 5$ but the \gb is $\langle z_1 + z_2 - 2, z_2 - 1 \rangle$.
This is coherent with the stabilization of
$\pi_4 (V(I)) = \pi(V(I))$.
\end{example}

\subsection{Eigenvalues of multiplication maps} \label{sec:eigenstructure}

On the top of computing a Gr\"obner bases for our ideal, we can also recover information from $\pi_b(V(I))$ by computing the eigenvalues of the multiplication maps.
We start this section by characterizing the eigenvalues of these maps and then we study their multiplicities.

\begin{theorem}
\label{thm:characteristic_polynomial}
    Let $I \subset R$ be a bihomogeneous ideal and let $(a,b) \in \mathbb{Z}^2$ an admissible bidegree. For every $g' \in (\RXh)_0$,
    there are $\mu_{\xi} > 0$, for $\xi \in \pi_b(V(I))$, such that
    the characteristic polynomial of {$m_{g'}$} is 
    \begin{equation}\label{eq:characteristic_polynomial}
    \mathrm{CharPol}_{m_{g'}}(\lambda) =
    \prod_{\xi \in \pi_b(V(I))} \Big(\lambda - g'(\xi)\Big)^{\mu_{\xi}}.\end{equation}
\end{theorem}

\begin{proof}
Observe that $\lambda$ is an eigenvalue, if and only, if $0 = m_{g'}(f') - \lambda f' = m_{g' - \lambda}(f')$, for some $f' \in (\RXh)_0$.
If we write $g' = \frac{g}{h^k}$, the previous holds if and only if $\bar{m}_{g - \lambda h^k}$ has a non-trivial kernel.
By Lemma~\ref{lem:invertibleIfRank}, this is equivalent to the fact that \begin{equation}\label{eq:eigenvalue}(I,g - \lambda h^k )_{a+k,b} \neq R_{a+k,b}.\end{equation}
Using the same argument as in Lemma~\ref{lem:characterisationOfGeneric}, which corresponds to the case where $k = 1$, \eqref{eq:eigenvalue} is equivalent to the fact that
$g - \lambda h^k$ is a zero-divisor on $R_X / (J_b : \mxinf)$, that is, $g - \lambda h^k$ vanishes at $\xi \in \pi_b(V(I))$.
This means that
$g(\xi) - \lambda h^k(\xi) = 0$, so $\lambda$ is as claimed.
\end{proof}

\begin{example}[Cont.~Example~\ref{ex:running3}] 
   Consider the admissible bidegree $(3,2)$ and the admissible form  $h = x_0 + x_1$. If we construct the multiplication map with respect to $\frac{x_1}{h}$ at this degree, its eigenvalues are $0$ and $1$, which correspond to the two minimal primes that do not contain $\mx$. However, the same construction at admissible bidegree $(2,4)$, leads to only one eigenvalue equal to $1$, which corresponds to the minimal prime that does not contain either $\mx$ or $\my$.
\end{example}

Once we have identified the eigenvalues of $m_{g}$, we would like to know the multiplicity of each of the zeros of the characteristic polynomial, i.e. find the values of the exponents $\mu_{\xi}$ in \eqref{eq:characteristic_polynomial}. In the zero-dimensional case, these values are exactly the algebraic multiplicities of each of the points $\xi \in V(I)$, i.e. length of the modules $(R/I)_{\mathfrak{p}_{\xi}}$.
In our case, we will find an analogue of such multiplicities by considering the length of the localization of the $R_X$-module $(R/I)_{\ast,b}$ at each of the primes $\mathfrak{p}_{\xi}$ for $\xi \in \pi_b(V(I))$.

\begin{definition}
\label{def:localizedModules}
Let $I \subset R$ be a bihomogeneous ideal and let $(a,b) \in \mathbb{Z}^2$ an admissible bidegree. Consider  the $R_X$-module $M_b := (R/I)_{\ast,b}$. Then, for every $\xi \in \pi_b(V(I))$, we consider:
$$M_{b,\xi} := M_b \otimes_{R_X} {(R_X)}_{\mathfrak{p}_{\xi}},$$
 {where $(R_X)_{\mathfrak{p}_{\xi}}$ is the localization of $R_X$ at the prime ideal ${\mathfrak{p}_{\xi}}$}.
\end{definition}

\begin{lemma}
\label{lem:primDecom}
    Let $I$ be a bihomogeneous ideal and let $(a,b) \in \mathbb{Z}^2$ an admissible bidegree. Then, the natural map defines an isomorphism:
    \begin{equation} (M_b)_{a} \xrightarrow[]{\cong} \prod_{\xi \in \pi_b(V(I))} M_{b,\xi} \end{equation}
\end{lemma}

\begin{proof}
    Consider the exact sequence of local cohomology of $M_b$ {(see \cite[Theorem  A4.1]{eisenbud1995})}: 
    \begin{equation}\label{eq:ses}0 \xrightarrow[]{} H^0_{\mx}\big(M_b\big) \xrightarrow[]{} M_b \xrightarrow[]{} \bigoplus_{v \in \mathbb{Z}}\Gamma(\P^n,\widetilde{M_b}(\nu)) \xrightarrow[]{} H^1_{\mx}\big(M_b\big) \xrightarrow[]{} 0 \end{equation}
where $\widetilde{M_b}$ is the coherent sheaf associated to the module $M_b$ and $\bigoplus_{\nu \in \mathbb{Z}}\Gamma(\P^n,\widetilde{M_b}(\nu))$ are its global sections. 

As $M_{b}$ is supported in a finite number of points, the global sections of $\widetilde{M_b}(\nu)$ {satisfy} $\Gamma(\P^n,\widetilde{M_b}(\nu)) \cong \prod_{\xi \in \pi_b(V(I))}M_{b,\xi}$ for any $\nu \in \mathbb{Z}$.
Thus, taking the graded pieces of degree $a$ {in \eqref{eq:ses}}, we have:   
\begin{equation}
0 \xrightarrow[]{} 
H^0_{\mx}(R/I)_{a,b} \xrightarrow[]{} 
(R/I)_{a,b} \xrightarrow[]{} \\ 
\prod_{\xi \in \pi_b(V(I))}M_{b,\xi} \xrightarrow[]{}
H^1_{\mx}(R/I)_{a,b}
\xrightarrow[]{} 0
\end{equation}
where the central arrow is the natural map.
This follows from the fact that $H_{\mx}^i(M_b)_a \cong H_{\mx}^i(R/I)_{a,b}$ for every $i \geq 0$, which follows by definition of local cohomology from Cech complexes; see for instance \cite[Lemma 3.7]{chardinholandator}. Thus, it is enough to show that $H_{\mx}^0(R/I)_{a,b} = H_{\mx}^1(R/I)_{a,b} = 0$.

Firstly, we have $H_{\mx}^0(R/I)_{a,b} = \big((I:\mxinf)/I\big)_{a,b} = 0$ by Corollary  \ref{cor:saturatedAtAdmDeg}.
Moreover, as we have $\HF_{R/I}(a,b) = \HF_{R/I}(a + k,b)$ for all $k > 0$, the Hilbert function of $M_b$ at degree $a$ coincides with its Hilbert polynomial. 
Thus, we can use Grothendieck-Serre formula \cite[§4.4]{bruns1998cohen} to deduce that $H_{\mx}^1((R/I)_{\ast,b})_a = 0$.
\end{proof}

{As the prime ideal $\mathfrak{p}_\xi$ define points in $R_X$, the localizations $(R_X)_\mathfrak{p_\xi}$ are local rings with residue field $k$. Therefore, the modules $M_{b,\xi}$ have a structure of $k$-vector space whose dimension is equal to their length}. Our next theorem characterizes the algebraic multiplicities of the multiplication matrices.

\begin{theorem}[{Cont.~of~Theorem ~\ref{thm:characteristic_polynomial}}]
\label{thm:characteristic_polynomial2}
    The exponents of characteristic polynomial of {$m_{g'}$} in Equation~\eqref{eq:characteristic_polynomial} are
    $\mu_\xi = \Length(M_{b,\xi})$, for $\xi \in \pi_b(V(I))$.
\end{theorem}

\begin{proof}
The proof follows the same argument as in \cite[§2, Prop. 2.7]{coxlitosh}. Multiplication by $g'$ and the isomorphism in Lemma \ref{lem:primDecom} induce a commutative diagram:
    $$
\begin{tikzcd}
(M_b)_{a} \arrow[r, "\cong"] \arrow[d, "\times g'"'] & \prod\limits_{\xi \in \pi_b(V(I))} \!\!\!\! M_{b,\xi} \arrow[d, "\times g' "] \\
(M_b)_{a} \arrow[r, "\cong"] & \prod\limits_{\xi \in \pi_b(V(I))} \!\!\!\! M_{b,\xi}
\end{tikzcd}
$$
Thus, the multiplication map {$m_{g'}$} has a block structure, where each block has dimension equal to the length of $M_{b,\xi}$  and corresponds to the map:
$$M_{b,\xi} \xrightarrow[]{\times g'} M_{b,\xi}.$$
Note that $\mathfrak{p}_{\xi}$ is the only associated prime in $M_{b,\xi}$. So, if $g' = \frac{g}{h^k}$, then, using the same argument as in Theorem  \ref{thm:characteristic_polynomial}, $g - \lambda h^k$ is a nonzero divisor in $M_{b,\xi}$ for some $\lambda \in k$, if and only if, $g - \lambda h^k \in \mathfrak{p}_{\xi}$. Therefore, $\lambda = \frac{g}{h^k}(\xi)$ which is the only possible eigenvalue in this block. 
\end{proof}

\section{Complexity bounds}
\label{sec::bounds}

We provide bounds for the admissible bidegrees $(a,b)$. 
These lead to an upper bound for the complexity of constructing multiplication maps and computing \gbs in our setting.

\begin{theorem}(Macaulay bound for finitely many points) \label{thm:macaulayBound}
Given an ideal $I = \langle f_1,\dots,f_l \rangle \subset R$  with $\deg(f_i) = (a_i,b_i)$ and such that $V(I) \subset \mathbb{P}^n \times \mathbb{P}^m$ defines a \emph{finite number of points}. Then, every bidegree $(a,b)$ such that $(a,b) \geq \sum_{i = 1}^l(a_i,b_i) - (n,m)$, $(a,b)$ is an admissible bidegree for $I$.
\end{theorem}

\begin{proof} The proof uses standard techniques, based on the comparison of two spectral sequences associated to the Koszul-Cech bicomplexes of $I$ and of $(I,h)$, where $h \in R_{1,0}$; see \cite[Prop.~3.15]{busechardinnemati} for the details. The first step is to deduce from these spectral sequences that all the local cohomology modules of both $R/I$ and $R/(I,h)$ vanish at bidegrees which are higher but not equal than
$\sum_{i = 1}^l(a_i,b_i) + (1,0) - (n+1,m+1).$
From there, applying the Grothendieck-Serre formula \cite[§ 4.4]{bruns1998cohen}, we deduce that for all such bidegree $(a,b)$ the Hilbert functions of both $R/I$ and $R/(I,h)$ are equal to their respective Hilbert polynomials, hence are both constants, which implies that $(a,b)$ is admissible.
\end{proof}

The approach proposed in this paper, together with the previous bound, can be further optimized to compute \gbs for \emph{zero-dimensional ideals} in product of projective spaces; see \cite{mixedgrobnerbasistowards}.

Our proof only works when $V(I)$ is zero-dimensional.
We performed several experiments in more general contexts, that is, when $V(I)$ is non-zero-dimensional,
and in all of them the Macaulay bound was an admissible bidegree.
Unfortunately, we were not able yet to prove that the bound holds for arbitrary systems, {neither to find a counter-example}. Therefore, in what follows, we present a drastically more pessimistic bound to obtain an admissible bidegree. 

In the next theorem, given $\ell \in \mathbb N$, let $N_\ell := \HF_{k[y_0,\dots,y_m]}(\ell)$, that is, the number of monomials of degree $\ell$ in the $\y$'s. 

\begin{theorem} \label{thm:generalBoundAdmDeg}
Given an ideal $I = \langle f_1,\dots,f_l \rangle \subset R$  with $\deg(f_i) = (a_i,b_i)$ and such that $\pi(V(I)) \subset \mathbb{P}^n$ defines a 
finite number of points. 
Then, if $(a,b)$  is such that
$b \geq (\sum_{i = 1}^lb_i) - m$ and 
$a \geq (\sum_{i=1}^l a_iN_{b-b_i}) - n$, it is an admissible bidegree for $I$. 
\end{theorem}

\begin{proof} For any $b \in \mathbb{N}$, we define the graded $R_X$-module $M_b := (R/I)_{\ast,b}$. Since $h$ is of bidegree $(1,0)$, it is independent on the $\y$'s and we have 
$$M_b/h M_b \simeq (R/(I,h))_{(*,b)}.$$
In particular, applying Theorem \ref{thm:macaulayBound} we get that the Hilbert function of the $R_X$-module $M_b/h M_b$ in degree $a$ is equal to zero if
$(a,b) \geq \sum_{i = 1}^l(a_i,b_i) - (n,m)$.
This implies that for all such $b$ the Hilbert polynomial of $M_b$ is a constant, which means that $M_b$ is geometrically supported on finitely many points of $\mathbb{P}^n$. Now, it remains to estimate a degree (in the $\x$'s) at which we are sure that this Hilbert function stabilizes at a constant value. 

For that purpose, we define the map $$ \phi_b : E:=\oplus_{i=1}^l R_X(-a_i)^{N_{b-b_i}} \xrightarrow {f_1,\ldots,f_l} F:={R_X}^{N_b},$$
which is a finite free presentation of the graded $R_X$-module $M_b$. We denote by $e$, respectively $f$, the rank of the free module $E$, respectively $F$ and we consider the generalized Koszul complex associated to $\phi_b$ which is of the following form (see \cite[Ap.~C.2]{northcott}):
$$ 
0\rightarrow S_{e-f-1}(F^*)\otimes \wedge^{e} E \rightarrow \cdots \rightarrow S_0(F^*)\otimes \wedge^{f+1} E \rightarrow  
E \xrightarrow{\phi_b} F,
$$
where $S(-)$ and $\wedge(-)$ stand for the symmetric and exterior algebras, respectively. An important property of this complex is that its homology is annihilated by the $0^{\textrm{th}}$ Fitting ideal of $\phi_b$, i.e.~the ideal generated by the maximal minors (of size $f$) of $\phi_b$. 
As a consequence, the homology of the complex is supported on the support of $M_b$ which are finitely many points. Now, to estimate the Hilbert function of $M_b$, which is the cokernel of $\phi_b$, we proceed with the comparison of the two spectral sequences associated to bicomplex obtained from the above generalized Koszul complex by replacing each term by its Cech complex (this is similar to what we did in the proof of Theorem \ref{thm:macaulayBound}, the Koszul complex being replaced by a generalized Koszul complex). It follows that the local cohomology modules of $M_b$ (actually, only the local cohomology modules $H^0_{\mx}(M_b)$ and $H^1_{\mx}(M_b)$ are nonzero) vanish at all degree $\nu$ such that $(\wedge^e E)_{\nu+n}=0$, i.e.,
$$ \nu \geq \Big( \sum\nolimits_{i=1}^l a_iN_{b-b_i} \Big)- n.$$ 
Therefore, the Hilbert function of $M_b$ at those degree $\nu$ is a constant, which concludes the proof (use Theorem \ref{thm:macaulayBound} for the bound for the stabilization of the Hilbert function of $R/(I,h)$ where $h$ is an admissible linear form).
\end{proof}

\begin{example}[Cont.~Example~\ref{ex:running3}]
\label{ex:running4} In order to illustrate how much room for improvement the Theorem  \ref{thm:generalBoundAdmDeg} leaves, in this example $(2,2)$ is already an admissible bidegree, but the upper bound is $(83,6)$.
\end{example}

\begin{proposition} \label{prop:PiBEqualsPi}
    If $b \geq (\sum_{i = 1}^l b_i) - m$, then $\pi_b(V(I)) = \pi(V(I))$ as a finite set of points.
\end{proposition}

\begin{proof}This is a classical result. As already observed in the proof of Lemma \ref{lem:zeroesOfJb}, we have inclusions $$J_{b}=\mathrm{Ann}_{R_X}((R/I)_{*,b}) \subset J_{b+1}=\mathrm{Ann}_{R_X}((R/I)_{*,b+1})$$
for all $b$. 
It is well known (and easy to check) that this increasing sequence of annihilators stabilizes as soon as $H^0_{\my}(R/I)_{(*,b)}=0$, and that this latter equality holds for all $b \geq  (\sum_{i = 1}^l b_i) - m$ (see for instance \cite[Prop.~3.18 and Theorem ~3.29]{BCP23}).
\end{proof}

As a consequence, we deduce a single exponential bound for the number of arithmetic operations for computing a \gb of $J_b$. We remark that, if the characteristic of the field is big enough, then $x_0$ is always an admissible form after performing a generic linear change of coordinates.

\begin{theorem}\label{thm:singleExponentialBound}
Let $I = \langle f_1,\dots,f_l \rangle \subset R$ be a bihomogeneous ideal such that $\pi(V(I))$ is finite and $x_0$ is an admissible linear form. The number of arithmetic operations needed to recover a \gb for $(J_b \otimes_{R_X} R_{X,x_0})_0 \subset k[\frac{x_1}{x_0},\dots,\frac{x_n}{x_0}]$, which describes
the variety $\pi(V(I))$, is upper bounded by
$$O(l^{4(n + m + nm)} A^{4n} B^{4(m + mn)}),$$
where $(A,B) = (\max_j \deg_x f_j, \max_j \deg_y f_j)$.
\end{theorem}

\begin{proof}
    Let $(a,b)$ be the admissible bidegree from Theorem ~\ref{thm:generalBoundAdmDeg}. Then, the number of rows/columns of the multiplication map $M_{{x_i}{/x_0}}$ is at most $S := \frac{l^{n + m + nm} A^n B^{m + mn}}{n! \, (m!)^{n+1}}$. Hence, we can construct each matrix in $O(S^4)$ arithmetic operations. We recover the \gb using Corollary ~\ref{cor:FGLM} in $O(nS^4)$ arithmetic operations; the complexity bound follows from the classical analysis of FGLM \cite{fglm}, taking into account that we work with linear relations of matrices and not vectors. 
\end{proof}

Better complexity estimations and faster algorithms can be obtained by using randomized approaches in FGLM; see e.g. \cite{faugere2017sparse}.

\subsubsection*{How to tell apart $\pi(V(I))$ from $\pi_b(V(I))$}\label{sec:verification}
The new methods proposed might compute the points in $\pi(V(I))$ and some extra ones. 
In this section, we study how to tell them apart.
For this, given $\xi \in \P^n$, we consider the homogenoeus ideal $I_{\xi} \subset R_Y := \bar{k}[y_0,\dots,y_m]$ obtained by specializing the variables $\x$'s to~$\xi$.

\begin{proposition}[{\cite[Corollary~3.30]{BCP23}}]\label{prop:verification}
Given an ideal $I = \langle f_1,\dots,f_l \rangle \subset R$  with $\deg(f_i) = (a_i,b_i)$. Given $\xi \in \mathbb P^n$, we have that $\xi \notin \pi(V(I))$, if and only if, $(I_\xi)_b = (R_Y)_b$ for $b = \sum^l_{i=1} b_i - m$.
\end{proposition}

To apply the Proposition \ref{prop:verification}, we need to take special considerations according to the computation paradigm in which we work.

\paragraph{Symbolic paradigm.} Using Corollary ~\ref{cor:FGLM}, we can compute a \gbs bases and from it a rational univariate representation of $\pi_b(V(I))$ \cite{demin2024reading} given by polynomials
$h_0(t),\dots,h_n(t),r(t) \in k[t]$, with $r(t)$  square-free, defining the superset of  $\pi(V(I))$, 
\linebreak
$\{ [h_0(\xi) \! :\! \cdots \! : \! h_n(\xi)] \in \mathbb P^n : \xi \in \bar{k}, r(t) = 0 \} .$
Let $r(t) = r_1(t) \cdots r_k(t)$ be the irreducible factorization of $r(t)$ over $k[t]$. 
For each $i$, consider 
$$W_i:=\{ [h_0(\xi) : h_1(\xi) : \dots : h_n(\xi)] \in \mathbb P^n : \xi \in \bar{k} \text{ and } r_i(t) = 0 \}.$$
Consider the homomorphism $\phi : R \rightarrow Q_i := k[y_0,\dots,y_m][t]/r_i(t)$, sending each $x_i \mapsto \phi(x_i) :=  h_i(t) \in {Q_i}$ and
$y_j \mapsto y_j$.
Prop.~\ref{prop:verification} implies that $W_i \subset \pi(V(I))$ if and only if we have the equality of $(k[t]/r_i(t))$-vector spaces
$(\phi(I))_b = (Q_i)_b$.

\paragraph{Numerical paradigm.} If $\bar\xi$ is an approximation of $\xi \in \mathbb \P^n$ and we want to verify if $\xi \in \pi(V(I))$, we need to check numerically if $(I_{\bar{\xi}})_b= (R_Y)_b$. This can be done by computing the numerical rank of the Macaulay matrix associated to $(I_{\bar{\xi}})$ at degree $b$. This kind of check has been employed in the context of Computer Aided Design \cite{buse_cad} and similar considerations apply to our context.

\section{Admissible bidegrees and Gr\"obner bases}
\label{sec:grobner}
In the previous sections, we have claimed that an arbitrary Gr\"obner basis of $I$ leads us to admissible degrees and thus to the construction of our multiplication maps.
Indeed, we can do this by exploiting  the \gb to verify if a bidegree is admissible.
For example, we can use the shape of the \gb to deduce the behavior of the Hilbert function \cite[Theorem~15.3]{eisenbud1995} or to construct smaller Macaulay matrices that allow us to perform rank computations efficiently \cite{bardet2015complexity}.

However, there is a particular \emph{degree reverse lexicographical monomial order} from which it is easier to deduce the admissible degrees.
This order satisfies the following inequalities,
\begin{equation}\label{eq:monOrder}
y_n > \dots > y_0 > x_n > \dots > x_0.
\end{equation}
If we work in generic coordinates, the variable $x_0$ is an admissible form. Hence, checking the saturation of the ideal with respect to $x_0$ or updating the \gbs of the ideal when we add $x_0$ can be done with no additional computational cost; see~\cite[Section~2]{bayer_criterion_1987}.

We had recently shown that, in generic coordinates, the structure of the \gb of a bihomogeneous ideal with respect to this order depends on some cohomological properties of the ideal.  In what follows, we will relate them to the admissible degrees, to show how the admissible degrees coincide with certain degrees of the minimal elements in the \gbs with respect to this order in generic coordinates, that is, with minimal generators of the bigeneric initial ideal \cite{aramova2000bigeneric}.
This illustrates the connection between the result of this paper and the computation of \gbs.

\begin{definition}
\label{def:partial-regularity}
Let $I \subset S$ be a bihomogeneous ideal. 
The \emph{partial regularity region}, $\xreg(I)$,
is a region of bidegrees $(a,b) \in \mathbb{Z}^2$, such that for all $i \geq 1$ and $(a',b') \geq (a - i + 1, b)$:
$$ H_{\mx}^i(I)_{(a',b')} = 0 . $$ 
\end{definition}

The region $\xreg(I)$ relates to the admissible degrees as follow.

    \begin{theorem}
    \label{thm:xRegularity}
      Let $I \subset R$ be a bihomogeneous ideal such that $\pi(V(I))$ is finite. Then, for $(a,b) \in \mathbb{N}^2$, the following are equivalent:
    \begin{itemize}
       \item[i)] $(a,b')$ is an admissible bidegree for all $b' \geq b$.
        \item[ii)] $(a,b) \in \xreg(I)$ and $\HF_{R/(I,h)}(a,b) = 0$.
    \end{itemize}
\end{theorem}

\begin{proof}
 From \cite[Theorem  4.7]{bender2024multigradedcastelnuovomumfordregularitygrobner}, we deduce that if $(a,b) \in \xreg(I)$, then $(I:h)_{(a',b')} = I_{(a',b')}$ for every $(a',b') \geq (a,b)$ and $h$ generic. From this, we deduce that $ii)$ implies $i)$. 

 Conversely, using Theorem  \ref{thm:stabilizationOfAdmDeg}, we deduce that if $(a,b')$ is an admissible bidegree for $b' \geq b$, then
    $(I:h)_{a',b'} = I_{a',b'}$ for any $(a',b') \geq (a,b)$. Together with the fact that $(I,h)_{a,b} = R_{a,b}$ and \cite[Theorem  4.7]{bender2024multigradedcastelnuovomumfordregularitygrobner}, this implies that $(a,b)\in \xreg(I)$.
\end{proof}

As a consequence, we relate admissible bidegrees and the generators of the bigeneric initial ideal, that is, the ideal of leading monomials of the elements in a minimal \gb of $I$ in generic coordinates, which we will denote as $\bigin(I)$.

\begin{corollary}
\label{cor:GBConclusion}
Let $(a,b) \in \mathbb{Z}^2$ such that $(a,b')$ is an admissible bidegree for all $b' \geq b$. Then, there is no minimal generator of $\bigin(I)$ of degree $(a+1,b)$. Moreover, if there is $b' \geq b$ such that $(a-1,b')$ is not an admissible bidegree and $\HF_{(I,h)}(a-1,b) = 0$, then there is a minimal generator of $\bigin(I)$ of degree $(a,b_0)$ for some $b_0 \leq b$.
\end{corollary}

\begin{proof}
    Follows from Theorem  \ref{thm:xRegularity} and \cite[Theorem  5.3, Thm 5.4]{bender2024multigradedcastelnuovomumfordregularitygrobner}.
\end{proof}

\begin{example}[Cont. of Example~\ref{ex:running4}]
\label{ex:running5} If we compute the bigeneric initial ideal of $I$ with respect to the reverse lexicographical order with the relative order being Equation~\eqref{eq:monOrder}, the bidegrees of its generators are:
$$\{(0,1),(0,2),(1,1),(2,0),(2,1),(3,1)\}.$$
Every bidegree bigger than $(2,2)$ is admissible, as the Hilbert function stabilizes to the Hilbert polynomial.
\end{example}

\begin{remark}
    The order defined in Equation~\eqref{eq:monOrder} is not the one we should use to eliminate the $\bm{y}$ variables, but the ones in $\bm{x}$. Further work is needed to understand this phenomena.
\end{remark}

\bibliographystyle{abbrv}
\bibliography{main}

\end{document}